\documentclass[11pt]{amsart}
\pdfoutput=1
\usepackage{amssymb}
\usepackage{amsmath}
\usepackage{amsfonts}
\usepackage[usenames]{color}
\usepackage{graphicx}
\usepackage{array}

\makeatletter
 
 \@addtoreset{equation}{section}
\makeatother

\newtheorem{theorem}{Theorem}[section]

\theoremstyle{definition}
\newtheorem{definition}[theorem]{Definition}
\newtheorem{lemma}[theorem]{Lemma}

\newtheorem{remark}[theorem]{Remark}

\numberwithin{equation}{section}

\newcommand{\R}{\mathbb{R}}

\newcommand{\bibtitle}[1]{\emph{#1}}
\newcommand{\dfn}[1]{\textit{#1}}
\newcommand{\st}{\, | \;} 
\newcommand{\lk}{\mathrm{lk}}

\begin{document}

\title{Linear embeddings of $K_9$ are triple linked}

\author{Ramin Naimi}
\address{Occidental College, Los Angeles, CA 90041}
\email{rnaimi@oxy.edu}

\author{Elena Pavelescu}
\address{Oklahoma State University, Stillwater, OK 74078}
\email{elena.pavelescu@okstate.edu}

\thanks{This research was supported in part by NSF grant DMS-0905300.}

\subjclass[2000]{Primary 05C10, 57M25}

\date{\today}

\keywords{oriented matroid, spatial graph, intrinsically linked, linkless embedding, linear embedding, straight-edge embedding}

\begin{abstract} 
We use the theory of oriented matroids to  show that any linear embedding of $K_9$, 
the complete graph on nine vertices, into 3-space
contains a non-split link with three components.
This shows that Sachs' conjecture 
on linear, linkless embeddings of graphs,
whether true or false, 
does not extend to 3-links.
\end{abstract}

\maketitle

\section{Introduction}
In 1977, Brown~\cite{Br} showed that every {linear (straight-edge) embedding}
of $K_7$, the complete graph on 7 vertices, in $\R^3$ contains a trefoil.
In the early 1980's, Sachs~\cite{Sa1, Sa2}, and Conway and Gordon~\cite{CG}, showed that $K_6$ is \dfn{intrinsically linked}, i.e., every embedding of it in $\R^3$ (or $S^3$), linear or not, 
contains two disjoint cycles that form a nontrivial link.
Conway and Gordon also showed that $K_7$ is \dfn{intrinsically knotted}, i.e., every embedding of it in $\R^3$ contains a nontrivial knot.
Sachs conjectured that if a graph has a \dfn{linkless embedding}, 
i.e., an embedding in $\R^3$ with no non-trivial links,
then it has a linear linkless embedding.
As far as we know, this conjecture remains open.


In \cite{FNP}, Flapan, Naimi, and Pommersheim showed that 
$K_{10}$ is \dfn{intrinsically $3$-linked} (I$3$L),
i.e., every embedding of it in $\R^3$ contains a non-split 3-component link.
They also showed that $K_9$ is not I3L by illustrating
an embedding of $K_9$ that contains no non-split $3$-link.
This embedding, however, was not a linear embedding.
So it is natural to ask whether
there also exists a linear embedding of $K_9$ that contains no non-split $3$-link;
and, more generally,
whether Sachs' conjecture for 2-links holds for 3-links.
We see below that the answer is negative.

\begin{theorem}
\label{mainTheorem}
Every linear embedding of $K_9$ in $\mathbb{R}^3$ contains a non-split link with three components.
\end{theorem}

It is not difficult to show that every finite graph has only finitely many linear embeddings up to isotopy.
So, in theory, one could check every linear embedding of $K_9$ to see whether or not there is one with no non-split $3$-links.
Finding all linear embeddings of $K_9$ up to isotopy, however, is a different story.
We were inspired by the paper of Ramirez Alfonsin~\cite{RA} to use Oriented Matroid theory to approach this problem.
To every linear embedding of a graph one can in a natural way associate an oriented matroid.
The oriented matroid determines the embedding of the graph 
up to ``linear embedding isotopy''; 
in particular, it determines which pairs of disjoint cycles in the embedded graph are linked.
Linear embeddings of $K_9$, in particular, give non-degenerate uniform oriented matroids of rank 4 on 9 elements. 
And the set of all such oriented matroids has already been found by the use of computers \cite{Fi}.
This provided us with a great tool for working on the $K_9$ problem.
It turns out that there are over two billion such oriented matroids.
We wrote a Mathematica computer program, which we ran in parallel on 31 computers, for 36 hours on average, to go through these oriented matroids.

We would also like to address the issue of ``verifiability.'' 
Although our computer program is fully documented 
and available on the arXiv~\cite{NP-appendix} for download and verification,
it would be very time consuming for 
anyone to analyze and check the program for errors.
So we also offer the following evidence that the program works correctly.
\begin{enumerate}
\item
The number, and types, of non-split links
in linear embeddings of 
$K_6$ and $K_7$
have been studied by 
Hughes~\cite{Hu}, and Jeon et.\ al.~\cite{Je}, respectively.
We ran our program on OM$(6,4)$ and OM$(7,4)$, 
all uniform oriented matroids of rank~ 4 on 6 and 7 elements, respectively;
and we obtained the same results as 
\cite{Hu} and \cite{Je}. 
We have posted details of our program's computations and output
for OM$(6,4)$ and OM$(7,4)$
in a document titled
``Supporting Evidence: $K_6$ and $K_7$ ''  \cite{NP-appendix}.

\item
The database of oriented matroids maintained by Finschi~\cite{Fi}
lists all isomorphism classes of OM$(9,4)$.
As supporting evidence, 
we present the details of our program's computations and output for 
three ``randomly'' chosen isomorphism classes:
the 1st, 1000th, and 1000000th classes.
There are 256 oriented matroids in every isomorphism class;
for each acyclic oriented matroid in these three classes,
we list all non-split 3--links,
as well as details of how these 3--links are found by the computer program.
We have posted these details in a document titled 
``Supporting Evidence:  Three Isomorphism Classes in OM$(9,4)$''~\cite{NP-appendix}.
We have also posted a document titled
``Sample Manual Verification''~\cite{NP-appendix}.
In it we explain how any of these 3--links,
along with the program's computations leading to the 3--link,
can be fully verified by manual computation;
and we perform such a manual computation,
with all details included, for one of the oriented matroids.
The reader may randomly choose
any of the other listed oriented matroids and verify 
the non-split 3--links listed there.
Furthermore, anyone with access to Mathematica
can download our program from~\cite{NP-appendix},
have the program print out its computations
for any isomorphism class,
and then perform manual verifications of them.

\end{enumerate}


\section{Acknowledgements}
We thank Lew Ludwig for suggesting this problem to us; 
Lukas Finschi for his help on retrieving the oriented matroid database and for maintaing an awesome website on oriented matroids; 
and Sonoko Moriyama for sharing with us the oriented matroid database we needed. 
We are also grateful to the faculty of the mathematics department at Occidental College for offering to run our program on their computers, and to the referee who offered many helpful comments and suggestions.

\section{Background on Oriented Matroids}
Before giving a precise and detailed account of how oriented matroids are associated with linear embeddings of graphs,
we first give an informal explanation for the case of $K_9$.
Every linear embedding of a graph in $\R^3$ is determined solely by where its vertices are embedded.
In a linearly embedded \textit{complete} graph,
no four vertices are coplanar, since otherwise at least two edges would intersect each other in their interiors.
Suppose we are given a linear embedding $\Gamma$ of $K_9$.
We label the vertices $1, 2,  \ldots, 9$.
Given any four vertices $a_1, a_2, a_3, a_4$ with $a_i<a_{i+1}$,
we can form three vectors, $v_i = a_4-a_i$, $i=1,2,3$.
Since the four points $a_1, a_2, a_3, a_4$ are not coplanar,
$\{v_1,v_2, v_3\}$ is a basis for $\R^3$ and
the determinant of the matrix $M = [v_1 | v_2 | v_3]$ is nonzero.
We assign a $+$ or $-$ sign to $\{a_1, a_2, a_3, a_4\}$ according to whether $\det(M)$ is positive or negative.
Doing this for every set of four vertices of $\Gamma$ amounts to associating an oriented matroid to $\Gamma$: a uniform oriented matroid of rank $r$ on $n$ elements 
is an assignment of a $+$ or $-$ sign to every $r$-subset of the $n$ elements (called bases), subject to certain conditions (called ``chirotope conditions'').
It turns out that this set of $+$ and $-$ signs captures enough information for determining which pairs of triangles in $\Gamma$ are linked.
For three non-collinear points, we will refer to both the union of the three edges they determine and to their convex hull as a triangle.

A different but equivalent way to assign an oriented matroid to the embedding $\Gamma$ is via circuits instead of bases.
For each 5--subset of the 9 vertices of $\Gamma$,
either one of the vertices will be inside the tetrahedron determined by the remaining four vertices,
or two of the vertices will determine an edge that ``pierces'' the triangle determined by the remaining three vertices,
as in Figure~\ref{circuits}.
Accordingly, each 5--subset of the 9 vertices is given a 4--1 or a 3--2 partition.
This assignment of a 4--1 or a 3--2 partition to every 5-subset is sufficient for describing an oriented matroid of rank 4 on 9 elements: a uniform oriented matroid of rank $r$ on $n$ elements can be defined by assigning a partition to each of the $(r+1)$-subsets (called circuits) of its elements, subject to certain conditions (called ``circuit axioms'').
There is a procedure for obtaining the signs of the bases from the partitioned circuits, and vice versa.
Knowing exactly which edges pierce which triangles is sufficient for determining whether any two disjoint triangles are linked or not.
Thus, knowing the circuit partitions of the oriented matroid associated to $\Gamma$ is all we need to determine the links in it.
With this informal description of oriented matroids and their connection to linear embeddings and linking, we now move on to a more formal and detailed presentation.

We first give the definition via circuits.
A \textit{signed set} $X$
is an ordered pair of disjoint sets, written as  $(X^+, X^-)$.
The \textit{opposite} of a signed set $X$ is the signed set $-X = (X^-,X^+)$.
Thus, $(-X)^+ = X^-$, and $(-X)^- = X^+$.
The \dfn{underlying set} of $X$ is defined as $\underline{X}=X^+\cup X^-$.

An \textit{oriented matroid} $\mathcal{M}$ on a finite set $E$ is defined by a collection $\mathcal{C}$ of sets $C=(C^+, C^-)$ with $\underline{C}\subset E$,
called \textit{circuits} of $\mathcal{M}$, satisfying the following axioms:
\begin{enumerate}
\item for all $C_1\in \mathcal{C}$, $\underline{C_1}\ne \emptyset$ and $-C_1\in \mathcal{C}$; \hfill (symmetry) 
\item for all $C_1, C_2\in \mathcal{C}$, if $\underline{C_1}\subseteq \underline{C_2}$, then $C_1=C_2$ or $C_1=-C_2$; \hfill (incomparability)
\item for all  $C_1, C_2\in \mathcal{C}$ with $C_1\ne C_2$, if $e\in C_1^+\cap C_2^- $, then there exists $C_3\in \mathcal{C}$ such that $C_3^+\subset (C_1^+\cup C_2^+) \setminus \{e\}$ and $C_3^-\subset (C_1^-\cup C_2^-) \setminus \{e\}$. \hfill (weak elimination)
\end{enumerate}

There is a natural way to obtain a new oriented matroid from a given one by an operation called reorientation: 
For a subset $A\subset E$  consider a new collection of signed sets, 
$_{-A}\mathcal{C}=\{_{-A}C \st C \in \mathcal{C}  \}$, where $_{-A}C$ 
is the signed set with $(_{-A}C)^+ = (C^+\setminus A) \cup (C^-\cap A)$ and 
$(_{-A}C)^- = (C^-\setminus A) \cup (C^+\cap A)$.
The collection $_{-A}\mathcal{C}$ represents the set of circuits for an oriented matroid, denoted by $_{-A}\mathcal{M}$, 
which is said to be obtained from $\mathcal{M}$ by \textit{reorientation on the set $A$}. 
This operation defines an equivalence relation on the set of oriented matroids of fixed rank on a ground set $E$. 
Two oriented matroids on $E$,  $\mathcal{M}$ and $\mathcal{M'}$, belong to the same \textit{isomorphism class} if  there exists a set $A\subset E$ such that $\mathcal{M'}=_{-A}\mathcal{M}$.

Let $E$ denote a finite set of points spanning $\mathbb{R}^r$, and let $\mathcal{C}$ be a collection of non-empty signed sets $C=(C^+, C^-)$ with $\underline{C}\subset E$ such that
\begin{enumerate}
\item  no proper subset of $\underline{C}$ is the underlying set for an element in $\mathcal{C}$, 
\item  for all $x \in \underline{C}$, there exists $\alpha_x\in\mathbb{R}$ such that 
\[ \sum_{x\in \underline{C}}\alpha_x x = 0 \quad \mathrm{and}  \quad \sum_{x\in \underline{ C}}\alpha_x = 0   \]
\end{enumerate}
We partition each $C \in \mathcal{C}$ by $C^+=\{x\in C \st \alpha(x)>0  \}$, and $C^-=\{x\in C \st \alpha(x)<0  \}$. 
The set of circuits $\mathcal{C}$ defines an oriented matroid of rank $r+1$ on $E$.
We will say this oriented matroid is \dfn{induced} by $E$.
An oriented matroid is said to be \textit{acyclic} if for each of its circuits $C$, $C^{-}\ne \emptyset$ and $C^{+}\ne \emptyset$;
otherwise it is said to be \dfn{cyclic}.
Oriented matroids defined by conditions (1) and (2) above are acyclic, since if $\sum_{x\in \underline{C}}\alpha_x = 0$, 
not all $\alpha_x$ are positive and not all $\alpha_x$ are negative.
An oriented matroid $\mathcal{M}$ of rank $r$ with ground set $E$
is said to be \textit{uniform}
if the collection of the underlying sets of the circuits of $\mathcal{M}$
equals the collection of all $(r+1)$--subsets of $E$.
We denote by $\mathrm{OM}(n,r)$ 
the set of all uniform oriented matroids of rank $r$ on a ground set with $n$ elements.   
Note that conditions (1) and (2) do not necessarily give a uniform oriented matroid.

Let $E$ be a set of points spanning $\mathbb{R}^3$ 
such that no four points are coplanar.
(This is equivalent to saying $E$ is
the set of vertices of a linearly embedded $K_n$ with $n \ge 4$.)
For any five points $a,b,c,d,e$ in $E$,
if one of them is inside the tetrahedron determined by the other four,
then the picture determine by the five points looks as in Figure~~\ref{circuits}(b);
otherwise, it looks as in Figure~~\ref{circuits}(a).
Accordingly, we have one of the following two cases.

Case 1. None of the five points is inside the tetrahedron determined by the other four.
Then, after relabeling the points if necessary, 
we can assume
that $\{a,b,c \}_{CH}$
and $\{d,e\}_{CH}$
intersect in some point 
 $\alpha_a a + \alpha_b b + \alpha_c c  = \alpha_d d  +\alpha_e e$,
 where $\alpha_a  + \alpha_b  + \alpha_c = 1$ and 
$\alpha_d  + \alpha_e = 1$. (For a set $S \subset \R^3$,  $S_{CH}$ denotes the convex hull of $S$.)
It follows that the signed set $(\{a,b,c\},\{d,e\})$
satisfies condition~(2)
(after reversing the signs of $\alpha_d$ and $\alpha_e$).
We say the circuit determined by these five points
has a (3--2)--partition.

Case 2.
None of the five points is inside the tetrahedron determined by the other four.
Then, after relabeling the points if necessary, 
we can assume $\{a,b,c,d \}_{CH}$ contains $e$.
So, by a similar reasoning to Case~1,
the signed set $(\{a,b,c,d\},\{e\})$
satisfies condition~(2).
 We say the circuit determined by the five points has a (4--1)--partition.
 
 Note that, because of non-coplanarity, no proper, non-empty subset of $\{a,b,c,d,e\}$
 satisfies condition~(2). 
 Therefore, in both cases above, condition~(1) is also satisfied.
 So $E$ induces a uniform, oriented matroid of rank~4.
 

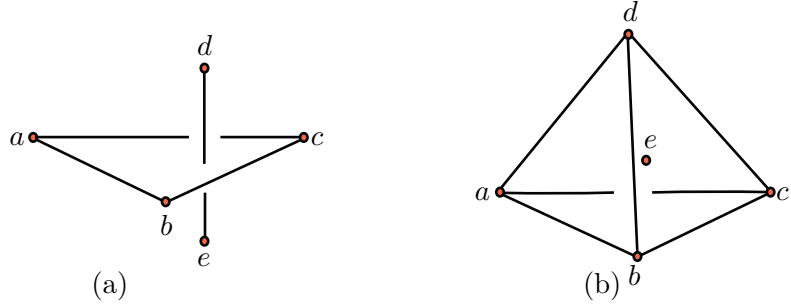
\begin{figure}[htpb!]
\begin{center}
\begin{picture}(280, 100)
\put(0,0){\includegraphics{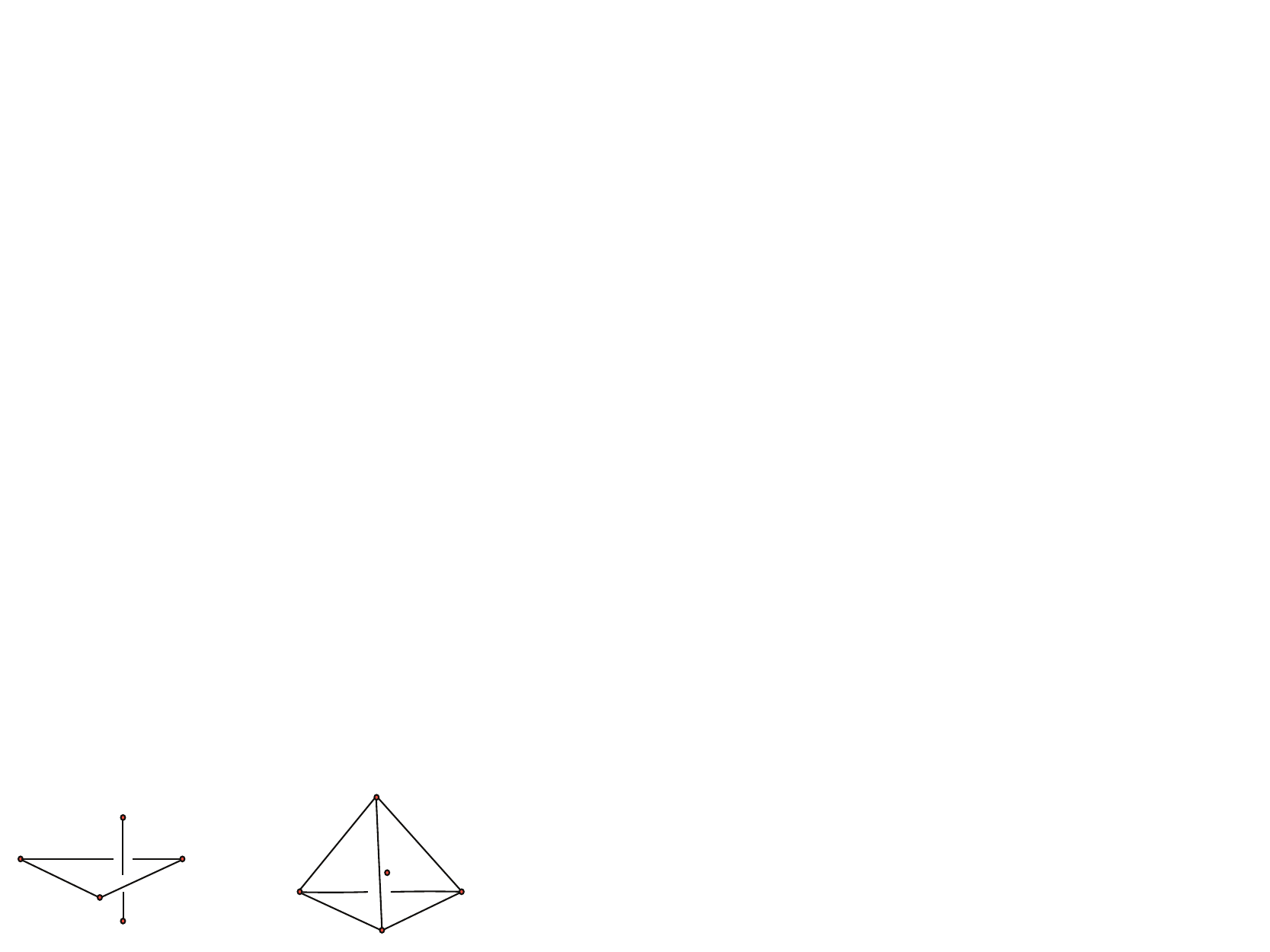}}
\put(4,56){$a$}
\put(118,56){$c$}
\put(61,22){$b$}
\put(75,10){$e$}
\put(75,90){$d$}
\put(35,0){(a)}
\put(221,0){(b)}
\put(180,35){$a$}
\put(238,3){$b$}
\put(294,35){$c$}
\put(236,103){$d$}
\put(244,55){$e$}
\end{picture}
\caption{(a) Circuit $(\{a,b,c\}, \{d,e\})$ . The edge $\{d,e\}_{CH}$ pierces the triangle $\{a,b,c\}_{CH}$. (b) Circuit $(\{a,b,c, d\}, \{e\})$. Point $e$ lies inside the tetrahedron $\{a,b,c,d\}_{CH}$.  }\label{circuits}
\end{center}
\end{figure}

\medskip

Oriented matroids can also be defined via basis orientations \cite{La}. 
For an integer $r\ge 1$ and a finite set $E$, an alternating map $\chi:E^{r} \rightarrow \{-1,0,1\}$ with certain properties called ``chirotope properties'' (\cite{BLSWZB}, p.124) gives a basis orientation for an oriented matroid of rank $r$ on $E$. 
The map $\chi$ specifies which $r$--subsets of $E$ are bases
by assigning a nonzero value to them, and assigns a sign to each subset which is a basis.
A chirotope $\chi:E^{r} \rightarrow \{-1,1\}$ defines  a \textit{uniform} oriented matroid (all $r$--subsets of $E$ are bases).
By ordering the elements of $E$, one can order the set of $r$--subsets of $E$.
A chirotope defining a uniform oriented matroid can be identified with a string of + and -- signs, of length $(_{\hspace{0.05in} r }^{|E|})$. 
The $k$th sign in the string represents the sign of the $k$th basis.
The string obtained by replacing every + with --, and vice versa, is defined to represent the same oriented matroid.

Given any uniform oriented matroid $\mathcal{M}$ of rank $r$ defined by a chirotope $\chi$,
we can obtain a collection $\mathcal{C}$ of circuits that define the same oriented matroid $\mathcal{M}$, as follows.
Let $\{x_1, x_2, \ldots, x_{r+1}\}$ be any ordered $(r+1)$--subset of the ground set $E$ of $\mathcal{M}$.
We define the sign of each $x_i$ by
\begin{equation}
\label{conversionEqn}
\textrm{sgn}(x_i) = (-1)^{i-1}\chi(\{x_1, \ldots, x_{i-1}, x_{i+1}, \ldots, x_{r+1}\})
\end{equation}
We let $C=(C^+, C^-)$, where $C^+=\{x_i |\, {\rm sgn}(x_i) >0 \}$ and $C^-=\{x_i |\, {\rm sgn}(x_i) <0 \}$.
Then, the collection $\mathcal{C}$ of all circuits $C$ defined this way will satisfy the circuit axioms 
and give the same oriented matroid as given by $\chi$.


\section{Proof of the Main Theorem}

Our proof of Theorem~\ref {mainTheorem} relies on the program which we describe at the end of this section.
The program relies on the following lemmas.


\begin{lemma}
Let $a,b,c,d,e,$ and $f$ be six distinct points in $\mathbb{R}^3$ such that no four are coplanar. 
Let $T$ be the triangle determined by $\{a,b,c\}$ (i.e., the boundary of $\{a,b,c\}_{CH}$) and  let $T'$ be the triangle determined by $\{d,e,f\}$. 
The following are equivalent:
\begin{enumerate}
\item The triangles $T$ and $T'$ are linked.
\item Exactly one edge of $T$  intersects $\{d,e,f \}_{CH}$.
\item Exactly one edge of $T'$  intersects $\{a,b,c \}_{CH}$.
\end{enumerate}
\label{Lemma-triangles} 
\end{lemma}

\begin{proof}
We prove $(1) \Leftrightarrow (2).$ By symmetry we get $(1) \Leftrightarrow (3).$

$(1) \Rightarrow (2).$ Assume the triangles $T$ and $T'$ are linked. 
Choose orientations for $\{a,b,c \}_{CH}$ and $\{d,e,f \}_{CH}$ and take the induced orientations on $T$ and $T'$, respectively.
 Since $\lk(T,T') \ne 0$  represents the signed sum of the intersections of  $T$ with   $\{d,e,f\}_{CH}$, at least one of the edges of $T$ intersects $\{d,e,f\}_{CH}$. 
 Assume exactly two edges of  $T$ intersect $\{d,e,f\}_{CH}$.
Since the two edges share a vertex, they intersect $\{d,e,f\}_{CH}$  with different signs and hence $\lk(T,T')= 0$, contradicting the hypothesis.
Three edges of $T$ cannot possibly intersect $\{d,e,f\}_{CH}$, since this would imply that the two planes determined by $\{a,b,c\}_{CH}$ and $\{d,e,f\}_{CH}$  share  three non-colinear points, and hence they coincide,  contradicting the non-planarity hypothesis. 

$(2) \Rightarrow (1).$ Assume $\{a,b \}_{CH}$ is the unique edge of $T$ intersecting $\{d,e,f\}_{CH}$.
Then $\lk(T,T') = \pm 1$, that is, $T$ and $T'$ are linked. 
\end{proof}


In light of the above lemma, the following definition gives a ``natural extension'' of the concept of linked triangles to oriented matroids.

\begin{definition}
Let $\mathcal{M}$ be an oriented matroid of rank 4 on a set $E$ with $|E|\ge 6$. 
To a pair of disjoint triangles $T= \{a,b,c\} \subset E$ and $T'=\{d,e,f\}\subset E$ we associate the sets $S=\{(\{a,b,c\}, \{d,e\}), (\{a,b,c\}, \{d,f\}), (\{a,b,c\}, \{e,f\}) \}$ and $S'=\{(\{d,e,f\}, \{a,b\}),(\{d,e,f\}, \{a,c\}), (\{d,e,f\}, \{b,c\}) \}$.
Then $T$ and $T'$ are said to be \textit{linked} if exactly one element of each of $S$ and $S'$ is a circuit of $\mathcal{M}$.  
\label{linkedtriangles}
\end{definition}


\begin{lemma}
Let $\mathcal{M}$ be an oriented matroid  on a set $E$ and  $A\subset E$. Then $_{-A}\mathcal{M}$ is cyclic if and only if  there exists  a circuit $C=(C^+, C^-)$ of $\mathcal{M}$ such that $A\cap \underline{C} = C^+$ or $A\cap \underline{C}= C^-$.
\label{cyclicreor}
\end{lemma}
\begin{proof}
By definition, 
$_{-A}\mathcal{M}$ is cyclic if and only if for some circuit $C_A=(C_A^+, C_A^-)$ of $_{-A}\mathcal{M}$ we have $C_A^+ = \emptyset$ or  $C_A^-=\emptyset$.
Now, $C_A = _{-A}C$ for some circuit $C$ of $\mathcal{M}$.
So we get: \\
$C_A^+ = \emptyset$ or $C_A^-=\emptyset $ \\
$\iff$ 
 $(C^+\setminus A) \cup (C^-\cap A) = \emptyset$ 
or
$(C^-\setminus A) \cup (C^+\cap A) = \emptyset$ \\
$\iff$ 
$C^+ \subset A$ and $C^-\cap A = \emptyset$ 
or
$C^-\subset A$ and $C^+\cap A = \emptyset$ \\
$\iff$ 
$A\cap \underline{C} = C^+$ or $A\cap \underline{C}= C^-$.

\end{proof}

To obtain the isomorphism class of an oriented matroid $\mathcal{M}$, it suffices to reorient only on sets  $A\subset E$ with $|A|\le \frac{1}{2}|E|$, since reorientation on $E\setminus A$ gives the same matroid as reorientation on $A$. 
For an oriented matroid on 9 elements,
there are $ \sum_{i=0}^4 {9\choose i} = 256$ such reorientation sets $A$.

Since an oriented matroid induced by a linear spatial graph is acyclic, 
our program only tests the acyclic oriented matroids of OM$(9, 4)$ for non-split 3--links.
For a given $\mathcal{M}$, we let $\mathcal{A}_{\mathrm{cyclic}}$ denote the collection of subsets  $A\subset E$ such that $_{-A}\mathcal{M}$ is cyclic. 
These are sets $A$ as in Lemma \ref{cyclicreor}.


\subsection{Description of the Mathematica Program.} 
Our program checks, as follows, that  in every acyclic $\mathcal{M}\in \mathrm{OM}(9,4)$  there are three disjoint triangles $T$, $T'$, and $T''$ such that $T$ links both $T'$ and $T''$.
(The following is described in greater detail in~\cite{NP-appendix}.)

\begin{enumerate}

\item 
Read in, from the given database of $\mathrm{OM}(9,4)$ isomorphism classes,
the oriented matroid $\mathcal{M}$ given for each isomorphism class;
$\mathcal{M}$ is given in chirotope form 
(as a string of ${9\choose4}=126$  + and  -- signs ).

\item
Compute the set of ${9\choose5}=126$ circuits for each $\mathcal{M}$ in (1),
as described in Equation~\ref {conversionEqn}. 

\item Using Lemma~\ref{cyclicreor}, compute the set $\mathcal{A}_{\mathrm{cyclic}}$. 

\item
For each $\mathcal{M}$,
and for each reorientation set $A$ not in $\mathcal{A}_{\mathrm{cyclic}}$,
compute the circuits of the oriented matroid $_{-A}\mathcal{M}$.

\item 
For each triangle $T$ in $K_9$,
find, from the (3--2)--partitioned circuits of $_{-A}\mathcal{M}$, 
the set of all edges that pierce $T$;
use this  and Lemma ~\ref{Lemma-triangles} to find the set of all triangles $T'$ that link $T$;
and use this in turn to determine if $T$ belongs to a non-split 3--link.

\end{enumerate}

There are 9,276,595 isomorphism classes in the set $\mathrm{OM}(9,4)$,
and each isomorphism class consists of 256 oriented matroids
(including cyclic and acyclic ones).
 Among these 2,374,808,320 oriented matroids, 
 the program checked all the acyclic ones and found a non-split 3-link in each one.

\begin{remark}
The program gives us a possibly stronger result than Theorem~\ref{mainTheorem}
since it found a non-split 3--link in all the acyclic oriented matroids in OM$(9,4)$.
However, we don't know which of these matroids are induced by an embedded $K_9$. 
As far as we know, there is no known systematic way of determining in general when an acyclic oriented matroid is induced by an embedding. 
\end{remark}


\bibliographystyle{amsplain}

\end{document}